\documentclass{amsart}%
\usepackage{hyperref}
\usepackage{amsmath}%
\usepackage{amsfonts}%
\usepackage{amssymb}%
\usepackage{amsthm}
\usepackage{mathrsfs}
\usepackage{todonotes}
\usepackage{bbm}
\usepackage{amscd}
\usepackage{youngtab}
\usepackage{tikz}
\usepackage{bm}
\usepackage{enumitem}
\usetikzlibrary{arrows}
\usetikzlibrary{matrix}
\usepackage{comment}

\newcommand{\ph}{\varphi}

\newcommand{\R}{\mathbb{R}}

\newcommand{\C}{\mathbb{C}}
\newcommand{\Z}{\mathbb{Z}}

\newcommand{\N}{\mathbb{N}}
\newcommand{\Hom}{\text{Hom}}

\newcommand{\mf}{\mathfrak}
\newcommand{\mc}{\mathcal}

\newcommand{\bs}{\boldsymbol}

\renewcommand{\ker}{\text{ker}}
\newcommand{\coker}{\text{coker}}

\newcommand{\id}{\text{id}}
\renewcommand{\dim}{\text{dim}}

\newcommand{\inj}{\hookrightarrow}
\newcommand{\ones}{\text{ones}}
\newcommand{\sur}{\twoheadrightarrow}
\newcommand{\Res}{\text{Res}}
\newcommand{\Ind}{\text{Ind}}
\newcommand{\Irr}{\text{Irr}}
\newcommand{\Cay}{\text{Cay}}
\newcommand{\diag}{\text{diag}}
\newcommand{\ch}{\text{ch}}

\renewcommand{\hat}{\widehat}
\renewcommand{\bar}{\overline}
\renewcommand{\tilde}{\widetilde}

\theoremstyle{definition}
\newtheorem{theorem}{Theorem}[section]
\newtheorem{def-prop}[theorem]{Definition-Proposition}
\newtheorem{prop}[theorem]{Proposition}
\newtheorem{example}[theorem]{Example}
\newtheorem{definition}[theorem]{Definition}
\newtheorem{conjecture}[theorem]{Conjecture}
\newtheorem{lemma}[theorem]{Lemma}

\theoremstyle{remark}
\newtheorem*{remark}{Remark}

\makeatletter
\def\blfootnote{\gdef\@thefnmark{}\@footnotetext}
\makeatother

\begin{document}
\title[Differential posets and restriction in critical groups]{Differential posets and restriction in critical groups}
\author{Ayush Agarwal}
\address{Stanford University, Stanford, CA.}
\email{ayush@stanford.edu} 
\author{Christian Gaetz}
\address{Department of Mathematics, Massachusetts Institute of Technology, Cambridge, MA.}
\email{gaetz@mit.edu}
\date{\today}

\begin{abstract}
In recent work, Benkart, Klivans, and Reiner defined the \textit{critical group} of a faithful representation of a finite group $G$, which is analogous to the critical group of a graph.  In this paper we study maps between critical groups induced by injective group homomorphisms and in particular the map induced by restriction of the representation to a subgroup.  We show that in the abelian group case the critical groups are isomorphic to the critical groups of a certain Cayley graph and that the restriction map corresponds to a graph covering map.  We also show that when $G$ is an element in a \textit{differential tower of groups}, as introduced by Miller and Reiner, critical groups of certain representations are closely related to words of up-down maps in the associated differential poset.  We use this to generalize an explicit formula for the critical group of the permutation representation of $\mf{S}_n$ given by the second author, and to enumerate the factors in such critical groups. 
\end{abstract}

\maketitle

\section{Introduction} \label{sec:intro}

\blfootnote{An extended abstract of this work appears in the proceedings of FPSAC 2018 \cite{abstract}.}

The critical group $K(\Gamma)$ is a well-studied abelian group invariant of a finite graph $\Gamma$ which encodes information about the dynamics of a process called \textit{chip firing} on the graph (see \cite{chip-firingsurvey} where critical groups are called \textit{sandpile groups}).  Recent work of Benkart, Klivans, and Reiner defined analogous abelian group invariants $K(V)$, also called \textit{critical groups}, associated to a faithful representation $V$ of a finite group $G$ \cite{BKR}.  It is known (see, for example, \cite{Tr}) that graph covering maps induce surjective maps between graph critical groups.  This paper investigates maps on critical groups of group representations which are induced by group homomorphisms.  

Differential posets, introduced by Stanley \cite{St}, generalize many of the combinatorial and enumerative properties of Young's lattice.  In \cite{MR}, Miller and Reiner introduced a very strong conjecture about the Smith normal form of $UD+tI$ where $U,D$ are the up and down maps in a differential poset, and $t$ is a variable.  We investigate how this conjecture, which was proven for powers of Young's lattice by Shah \cite{Sh}, can be used to determine the structure of critical groups in certain \textit{differential towers of groups.}

Section \ref{sec:defs} defines critical groups for graphs and group representations and gives background results.  It also discusses background on differential posets and differential towers of groups which will be used throughout the later sections.

In Section \ref{sec:maps} we study maps between critical groups which are induced by group homomorphisms.  In particular, restriction of representations to a subgroup $H \subset G$ induces a map $\bar{\Res}:K(V) \to K(\Res^G_H V)$.  When $G$ is abelian, Theorem \ref{thm:cay-graph} shows that $K(V)$ can be identified with the critical group of a certain Cayley graph $\Cay(\hat{G},\mc{S}_V)$, and that the restriction map $\bar{\Res}$ agrees with a map on graph critical groups induced by a natural graph covering.

In \cite{G1}, the second author determined the exact structure of the critical group for the permutation representation of the symmetric group $\mf{S}_n$.  This result depended on a relationship between tensor products with the permutation representation and the up and down maps in Young's lattice of integer partitions.  Section \ref{sec:gen-perm-rep} formalizes this connection and generalizes it to the context of differential towers of groups, allowing us to explicitly compute the critical group for a generalized permutation representation of the wreath product $A \wr \mf{S}_n$ in Theorem \ref{thm:gen-perm-rep}.  It also investigates properties of the critical groups associated to representations $V(w)$ which occur by repeatedly applying restriction and induction to the trivial representation in a differential tower of groups.  The pattern of restriction and induction is specified by a word $w \in \{U,D\}^*$, where $U,D$ are the up and down operators in the corresponding differential poset.  In Theorem \ref{thm:structure of K(V(f))} we show that the structure of the critical group $K(V(w))$ is closely related to combinatorial properties of the up and down operators, as studied in \cite{St}.    

Finally, in Section \ref{sec:factors}, Theorem \ref{thm:ones of w} gives an enumeration of the factors in the elementary divisor form of $K(V(w))$ in terms of the rank sizes of the corresponding differential tower of groups.  We conclude by presenting a conjecture for the size and multiplicity of the smallest nontrivial factor in $K(V(U^kD^k))$ in $A \wr \mf{S}_n$.  This conjecture also relates the larger factors in this critical group to the factors in a critical group for the subgroup $A \wr \mf{S}_{n-k}$.

\section{Background and definitions} \label{sec:defs}
\subsection{Critical groups of graphs}
This section gives some background on critical groups of graphs; see \cite{chip-firingsurvey} for a thorough survey.  We will be interested in critical groups of graphs primarily as motivation for our study of critical groups of group representations, however Section \ref{subsec:cayley} below gives a close relationship between the two concepts when $G$ is abelian. 

Let $\Gamma$ be a finite directed graph with a unique sink $s$, we sometimes designate a vertex as a sink and therefore ignore its outgoing edges.  Fix some ordering $s=v_0,v_1,...,v_{\ell}$ of the vertices of $\Gamma$, and let $d_i$ be the out-degree of $v_i$ and $a_{ij}$ be the number of edges from $v_i$ to $v_j$.  Then the \textit{Laplacian matrix} $\tilde{L}(\Gamma)$ has entries
\[
\tilde{L}(\Gamma)_{ij}=\begin{cases} d_i-a_{ii} & \text{for $i=j$,} \\ -a_{ij} & \text{for $i \neq j$.} \end{cases}
\]
The \textit{reduced Laplacian matrix} $L(\Gamma)$ is the $\ell \times \ell$ matrix obtained from $\tilde{L}(\Gamma)$ by removing the row and column corresponding to the sink.

The \textit{critical group} $K(\Gamma)$ (also called the \textit{sandpile group} in the literature), defined as 
\[
K(\Gamma) = \coker(L(\Gamma): \Z^{\ell} \to \Z^{\ell}),
\]
is a finite abelian group whose order is the number of spanning trees of $\Gamma$ which are directed towards the sink.  There are two sets of distinguished coset representatives for $K(\Gamma)$, the \textit{superstable} and \textit{recurrent} configurations, which encode the dynamics of a process called \textit{chip-firing} on $\Gamma$.  In the remainder of the paper we will primarily be interested in the group structure of $K(\Gamma)$, and not in these particular coset representatives or the chip-firing process.

Given two directed multigraphs $\Gamma, \Gamma'$, a \textit{graph map} is a continuous map $\ph: \Gamma \to \Gamma'$ of the underlying topological spaces which maps the interior of each edge homeomorphically to the interior of another edge, preserving orientation; by continuity this also defines a map from the vertices of $\Gamma$ to the vertices of $\Gamma'$.  A graph map is a \textit{graph covering} if in addition each vertex of $\Gamma$ has a neighborhood on which the restriction of $\ph$ is a homeomorphism.

The following proposition is well-known, see for example \cite{Tr}:

\begin{prop} \label{prop:graph covering surj}
The underlying map on vertices of a graph covering $\ph: \Gamma \to \Gamma'$ induces a surjective group homomorphism $\bar{\ph}: K(\Gamma) \sur K(\Gamma')$.
\end{prop}

Section \ref{subsec:cayley} discusses the relationship between maps induced on critical groups of Cayley graphs by certain graph coverings $\bar{\ph}$ and the map $\bar{\Res}$ on critical groups of group representations.

\subsection{Critical groups of group representations}
Let $G$ be a finite group and $V$ a faithful complex (not-necessarily-irreducible) representation of $G$; let $\mathbbm{1}_G=V_0, V_1,...,V_{\ell}$ denote the irreducible complex representations and $\chi_i$, $i=0,...,\ell$ denote their characters.  Let $R(G)$ denote the \textit{representation ring} of $G$.  This is the commutative $\Z$-algebra of formal integer combinations of representations of $G$ modulo the relations $[W \oplus W']=[W] + [W']$; the product structure is defined as $[W] \cdot [W'] = [W \otimes_{\C} W']$.  As a $\Z$-module, $R(G)$ is isomorphic to $\Z^{\ell+1}$, since the classes of irreducible representations $[\mathbbm{1}_G],[V_1],...,[V_{\ell}]$ form a basis.  We define elements 
\[
\delta^{(g)}=\sum_{W \in \Irr(G)} \chi_W(g)\cdot [W]
\]
of $R(G)$ corresponding to the columns in the character table of $G$.

The representation ring $R(G)$ is endowed with a $\Z$-algebra homomorphism $\dim: R(G) \to \Z$ sending representations $[W]$ to their dimensions as vector spaces (which we also denote by $\dim(W)$), and extending by linearity to virtual representations.  The kernel of this map, which we denote by $R_0(G)$, is the ideal of elements in $R(G)$ with virtual dimension 0.  

Then multiplication by the element $\dim(V)[\mathbbm{1}_G]-[V]$ defines a linear map $\tilde{C}_V:R(G) \to R(G)$.  Since $\dim(V)[\mathbbm{1}_G]-[V] \in R_0(G)$, this descends to a linear map
\[
C_V: R_0(G) \to R_0(G).
\] 
\begin{def-prop}[\cite{BKR}, Proposition 5.12] \label{def-prop:critical group}
If $V$ is a faithful finite dimensional representation of $G$, then the linear map $C_V$ is nonsingular, and so $\coker(C_V)$ is a finite abelian group.  We define the \textit{critical group} $K(V)$ to be this cokernel.  We also have that $\coker(\tilde{C}_V) = \Z \cdot \delta^{(e)} \oplus K(V)$.
\end{def-prop}

\begin{remark}
As a quotient of $R_0(G)$, the critical group $K(V)$ inherits a multiplicative structure in addition to its (additive) abelian group structure.  We are interested here only in the additive structure of $K(V)$.  See Sections 5 and 6 of \cite{BKR} for some discussion of the multiplicative structure.
\end{remark}

We will need the following facts about critical groups and the maps $\tilde{C}_V$:

\begin{prop}[\cite{BKR}, Proposition 5.3] \label{prop:eigenvectors}
A full set of orthogonal eigenvectors for $\tilde{C}_V$ is given by the column vectors $\delta^{(g)}$ in the character table for $G$: 
\[
\tilde{C}_V\delta^{(g)}=(\dim(V)-\chi_V(g))\delta^{(g)},
\]
where $g$ ranges over a set of conjugacy class representatives for $G$, and $\chi_V$ denotes the character of $V$.
\end{prop}

\begin{remark} 
When $V$ is faithful, $\chi_V(g) \neq \dim(V)$ for $g \neq e$, thus Proposition \ref{prop:eigenvectors} shows that $\ker(\tilde{C}_V)$ is spanned by $\delta^{(e)}=\sum_{i=0}^{\ell} \dim(V_i)[V_i]=[V_{reg}]$, where $V_{reg}$ is the regular representation of $G$.
\end{remark}

\begin{theorem}[\cite{G1}, Theorem 3]\label{thm:repeated}
Let $V$ be a faithful representation of a finite group $G$, then:
\begin{itemize}
\item[a.] Let $e=c_0,...,c_{\ell}$ be a set of conjugacy class representatives for $G$.  The order of the critical group is given by:
\begin{equation}
|K(V)|=\frac{1}{|G|}\prod_{i=1}^{\ell}(\dim(V)-\chi_V(c_i)).
\end{equation}
\item[b.] If $V$ is self-dual and $a$ is an integer value of $\chi_V$ achieved on $m$ different conjugacy classes of $G$, then $K(V)$ contains a subgroup isomorphic to $(\Z/(\dim(V)-a)\Z)^{m-1}$. 
\end{itemize}
\end{theorem}

\begin{example} \label{ex:perm rep of S4}
Let $G=\mf{S}_4$ and let $V=\C^4$ be the 4-dimensional representation where $G$ acts by permuting coordinates.  Working in the basis $\{[V_0],...,[V_4]\}$ of $R(G)$ given by the character table below, we decompose each tensor product $V \otimes V_i$ into irreducibles, giving the rows of the matrix $\tilde{C}_V$.  
\begin{center}
\begin{tabular}{|c|c|c|c|c|c|}
\hline
& $e$ & $(12)$ & $(123)$ & $(1234)$ & $(12)(34)$ \\ \hline \hline
$\chi_V$ & 4 & 2 & 1 & 0 & 0 \\ \hline
$\chi_0$ & 1 & 1& 1 & 1 & 1\\ \hline
$\chi_1$ & 3 & 1 & 0 & -1 & -1 \\ \hline
$\chi_2$ & 2 & 0 & -1 & 0 & 2 \\ \hline
$\chi_3$ & 3 & -1 & 0 & 1 & -1 \\ \hline
$\chi_4$ & 1 & -1 & 1 & -1 & 1 \\ \hline
\end{tabular}
\end{center}

\[
\tilde{C}_V = \begin{pmatrix}
3 & -1 & 0 & 0 & 0 \\ -1 & 2 & -1 & -1 & 0 \\ 0 & -1 & 3 & -1 & 0 \\ 0 & -1 & -1 & 2 & -1 \\ 0 & 0 & 0 &-1 & 3
\end{pmatrix}\]

To calculate the cokernel of $\tilde{C}_V: R(G) \to R(G)$, we compute the Smith normal form (see Section \ref{sec:snf} below) of $\tilde{C}_V$ to get $\diag(0,1,1,1,4)$.  This shows that $\coker(\tilde{C}_V)\cong \Z \oplus \Z/4\Z$, and so $K(V) \cong \Z/4 \Z$.

Alternatively, we could apply Theorem \ref{thm:repeated}(a) to see that 
\[
|K(V)|=\frac{1}{4!}(4-2)(4-1)(4-0)(4-0)=4
\]
and apply part (b) to see that $K(V)$ has a subgroup isomorphic to $\Z/4\Z$.  This forces $K(V) \cong \Z/4\Z$.

The critical groups for the permutation representation of $\mf{S}_n$ were computed by the second author in \cite{G1}.  In Section \ref{sec:gen-perm-rep} we generalize this result further. 
\end{example}

\subsection{Differential posets}
Differential posets are a class of partially ordered sets defined by Stanley in \cite{St}.  Differential posets retain many of the striking enumerative and combinatorial properties of Young's lattice $Y$, the lattice of integer partitions ordered by containment of Young diagrams.

\begin{figure}[ht] 
\begin{center}
\begin{tikzpicture}[shorten >=1pt, auto, node distance=3cm, ultra thick,
   node_style/.style={fill=white!20!,font=\sffamily\Large\bfseries},
   edge_style/.style={draw=black, thick}]

    \node[node_style, scale=1] (v0) at (5,-1) {$\emptyset$};
    \node[node_style, scale=0.7] (v1) at (5,0) {\yng(1)};
    \node[node_style, scale=0.7] (v2) at (4,1) {\yng(2)};
    \node[node_style, scale=0.7] (v3) at (6,1){\yng(1,1)};
    \node[node_style, scale=0.7] (v4) at (3,2){\yng(3)};
    \node[node_style, scale=0.7] (v5) at (5,2){\yng(2,1)};
    \node[node_style, scale=0.7] (v6) at (7,2){\yng(1,1,1)};
    \node[node_style, scale=0.7] (v7) at (2,3){\yng(4)};
    \node[node_style, scale=0.7] (v8) at (3.65,3){\yng(3,1)};
    \node[node_style, scale=0.7] (v9) at (5,3.15){\yng(2,2)};
    \node[node_style, scale=0.7] (v10) at (6.15,3.1){\yng(2,1,1)};
    \node[node_style, scale=0.6] (v11) at (7.8,3.1){\yng(1,1,1,1)};
    \node[node_style, scale=0.6] (v12) at (0.5,4.40){\yng(5)};
     \node[node_style, scale=0.6] (v13) at (2.15,4.40){\yng(4,1)};
     \node[node_style, scale=0.6] (v14) at (3.5,4.40){\yng(3,2)};
	\node[node_style, scale=0.6] (v15) at (4.94,4.40){\yng(3,1,1)};
    \node[node_style, scale=0.6] (v16) at (6.22, 4.45){\yng(2,2,1)};
    \node[node_style, scale=0.6] (v17) at (7.35,4.6){\yng(2,1,1,1)};
    \node[node_style, scale=0.6] (v18) at (8.8,4.6){\yng(1,1,1,1,1)};
    
    \draw[edge_style]  (v0) edge node{} (v1);
    \draw[edge_style]  (v1) edge node{} (v2);
    \draw[edge_style]  (v1) edge node{} (v3);
    \draw[edge_style]  (v2) edge node{} (v4);
    \draw[edge_style]  (v2) edge node{} (v5);
    \draw[edge_style]  (v3) edge node{} (v5);
    \draw[edge_style]  (v3) edge node{} (v6);
    \draw[edge_style]  (v4) edge node{} (v7);
    \draw[edge_style]  (v4) edge node{} (v8);
    \draw[edge_style]  (v5) edge node{} (v8);
    \draw[edge_style]  (v5) edge node{} (v9);
    \draw[edge_style]  (v5) edge node{} (v10);
    \draw[edge_style]  (v6) edge node{} (v10);
    \draw[edge_style]  (v6) edge node{} (v11);
 	\draw[edge_style]  (v7) edge node{} (v12);
    \draw[edge_style]  (v7) edge node{} (v13);
    \draw[edge_style]  (v8) edge node{} (v13);
    \draw[edge_style]  (v8) edge node{} (v14);
    \draw[edge_style]  (v8) edge node{} (v15);
    \draw[edge_style]  (v9) edge node{} (v14);
    \draw[edge_style]  (v9) edge node{} (v16);
    \draw[edge_style]  (v10) edge node{} (v15);
    \draw[edge_style]  (v10) edge node{} (v16);
    \draw[edge_style]  (v10) edge node{} (v17);
    \draw[edge_style]  (v11) edge node{} (v17);
    \draw[edge_style]  (v11) edge node{} (v18);
\end{tikzpicture}
\end{center}
\caption{Young's lattice $Y$, a 1-differential poset.}
\label{fig:youngs lattice}
\end{figure}

We refer the reader to \cite{EC1} for basic definitions related to posets in what follows.

\begin{definition}[\cite{St}, Definition 1.1 and Theorem 2.2]
For $r \in \Z_{>0}$, a poset $P$ is called an \textit{r-differential poset} if the following properties hold:
\begin{itemize}
\item[(DP1)] $P$ is a graded locally-finite poset with $\hat{0}$. 
\item[(DP2)] Let $\Z^{P_n}$ be the free abelian group spanned by elements of the $n$-th rank of $P$.  Define the \textit{up and down maps} $U_n:\Z^{P_n} \to \Z^{P_{n+1}}$ and $D_n:\Z^{P_n}\to \Z^{P_{n-1}}$ by 
\begin{align*}
&U_n x := \sum_{x \lessdot y} y, &D_n y := \sum_{x \lessdot y} x,
\end{align*}
where $x \lessdot y$ means that $y$ covers $x$.  Then we require that for all $n$ we have 
\begin{align*}
D_{n+1}U_n - U_{n-1}D_n = rI.
\end{align*}
\end{itemize}
\end{definition}
\noindent When the context is clear we omit the subscripts from the  up and down maps.

When $P$ is a differential poset, we let $p_n=|P_n|$ denote the size of the $n$-th rank, and we let $\Delta p_n = p_n-p_{n-1}$ denote the difference in the sizes of consecutive ranks.  We make the convention that $p_i=0$ for $i<0$.  In the case of Young's lattice, $p_n=p(n)$ where $p(n)$ denotes the number of integer partitions of $n$.  If $\lambda$ is a partition, we let $\lambda'$ denote the conjugate partition obtained by reflecting the Young diagram across the diagonal. The following results of Stanley characterize the eigenspaces of $UD$ in terms of the rank sizes.

\begin{theorem}[\cite{St}, Theorem 4.1] \label{thm:DP eigenvalues}
Let $P$ be an $r$-differential poset and let $n \in \N$.  Then $UD_n$ is semisimple and has characteristic polynomial:
\[
\ch(UD_n)=\prod_{i=0}^n (x-ri)^{\Delta p_{n-i}}.
\]
\end{theorem}

\begin{theorem}[\cite{St}, Proposition 4.6] \label{thm:DP eigenspaces}
Let $E_n(ri)$ denote the eigenspace of $UD_n$ belonging to the eigenvalue $ri$, then
\[
E_n(0)=\ker(D_n)=(UP_{n-1})^{\perp}
\]
and 
\[
E_n(ri)=U^iE_{n-i}(0)
\]
for $1 \leq i \leq n$.
\end{theorem}

\subsection{Differential towers of groups}
\begin{definition}[\cite{MR}, Definition 6.1]
For $r \in \Z_{>0}$ define an \textit{r-differential tower of groups} $\mf{G}$ to be an infinite tower of finite groups
\[
\mf{G}: \{e\}=G_0 \subseteq G_1 \subseteq G_2 \subseteq \cdots 
\]
such that for all $n$:
\begin{itemize}
\item[(DTG1)] The branching rules for restricting irreducibles from $G_n$ to $G_{n-1}$ are \\ multiplicity-free, and 
\item[(DTG2)] $\Res^{G_{n+1}}_{G_n} \Ind_{G_n}^{G_{n+1}}- \Ind_{G_{n-1}}^{G_n}\Res_{G_{n-1}}^{G_n} = r \cdot \id$
where both sides are regarded as linear operators on $R(G_n)$.
\end{itemize}
\end{definition}

An $r$-differential tower of groups $\mf{G}$ corresponds to an $r$-differential poset $P=P(\mf{G})$ whose $n$-th rank $P_n$ is in bijection with the set $\Irr(G_n)$ of irreducible representations of $G_n$.  We will use Greek letters like $\lambda$ to denote elements of $P(\mf{G})$ and $V_{\lambda}$ to denote the corresponding irreducible representation.  We write $|\lambda|=n$ if $\lambda \in P_n$, or equivalently if $V_{\lambda}$ is a representation of $G_n$.  For $\lambda \in P_n$ and $\mu \in P_{n+1}$, $\lambda \lessdot \mu$ in $P$ if and only if $\Res^{G_{n+1}}_{G_n} V_{\mu}$ contains $V_{\lambda}$ in its irreducible decomposition, thus condition (DTG2) becomes condition (DP2). 

\begin{example}
Let $Y$ denote Young's lattice of integer partitions (see Figure \ref{fig:youngs lattice} above).  It is well known that irreducible representations of the symmetric group $\mf{S}_n$ are indexed by partitions $\lambda=(\lambda_1, \lambda_2,...)$ with $|\lambda|=\sum_i \lambda_i=n$; we refer the reader to \cite{James} for background on the representation theory of the symmetric group.  Young's rule says that $\Res^{\mf{S}_n}_{\mf{S}_{n-1}} V_{\lambda}$ decomposes as a direct sum of $V_{\nu}$ where $\nu$ ranges over all possible ways to remove a single box from the Young diagram for $\lambda$.  It is well known \cite{St}  that $Y$ is a 1-differential poset, so (DP2) holds, and by the above identification (DTG2) also holds.  Thus
\[
\mf{S}: \{e\} \subset \mf{S}_1 \subset \mf{S}_2 \subset \cdots
\] 
is a 1-differential tower of groups, with $P(\mf{S})=Y$.

More generally, if $A$ is an abelian group of size $r$, then Okada \cite{O} showed that the tower of wreath products $A \wr \mf{S}: \{e\} \subset A \subset A \wr \mf{S}_2 \subset A \wr \mf{S}_3 \subset \cdots$ is an $r$-differential tower of groups with $P(A \wr \mf{S})=Y^r$.
\end{example}

In the following result we show that the groups in any differential tower of groups have the same order as those in $A \wr \mf{S}$.

\begin{prop} \label{prop: size of DTG}
Let $\mf{G}: \{e\}=G_0 \subset G_1 \subset \cdots$ be an $r$-differential tower of groups, then $|G_n|=r^n \cdot n!$ for all $n \geq 0$.
\end{prop}
\begin{proof}
Let $P=P(\mf{G})$ be the corresponding $r$-differential poset.  Since restriction of representations does not change dimension, we have 
\[
\dim(V_{\lambda})=\sum_{\mu \lessdot_P \lambda} \dim(V_{\mu}).
\]
It follows by induction that $\dim(V_{\lambda})=e(\lambda)$, where $e(\lambda)$ denotes the number of upward paths from $\hat{0}$ to $\lambda$ in the Hasse diagram of $P$.  Stanley showed in Corollary 3.9 of \cite{St} that for any $r$-differential poset $Q$ we have $\sum_{x \in Q_n} e(x)^2 = r^n \cdot n!$.  Applying this to $P$, and recalling that $\sum_{V \in \Irr(G)}\dim(V)^2=|G|$ for any finite group $G$ gives the desired result. 
\end{proof}

\begin{remark}
Recent work \cite{G2} of the second author shows that the tower $\mf{S}$ of symmetric groups is the only 1-differential tower of groups.
\end{remark}

\subsection{Smith normal form and cokernels of linear maps} \label{sec:snf}

The cokernel of a linear map over a PID is described by the Smith normal form of the corresponding matrix.  In this section we review basic facts about Smith normal form, and state a conjecture about the Smith normal form of the map $UD$ in a differential poset.

\begin{definition}
Let $M \in R^{n \times n}$ be an $n \times n$ matrix with entries in some ring $R$.  We say that $S$ is a \textit{Smith normal form} for $M$ if:
\begin{itemize}
\item $S=PMQ$ for some $P,Q \in GL_n(R)$, and 
\item $S$ is a diagonal matrix $S=\diag(s_1,...,s_n)$ such that successive diagonal entries divide one another: $s_i | s_{i+1}$ for all $i=1,...,n-1$.
\end{itemize}
\end{definition}

The following facts are well-known (see, for example \cite{StSNF}).

\begin{prop} \text{}
\begin{itemize}
\item[a.] If $M$ has a Smith normal form, then it is unique up to multiplication of the $s_i$ by units in $R$.
\item[b.] If $R$ is a PID, then all matrices $M$ have a Smith normal form.
\item[c.] If $M$ has a Smith normal form $S=\diag(s_1,...,s_n)$, then 
\[
\coker(M) \cong \bigoplus_{i=1}^n R/(s_i).
\]
\end{itemize}
\end{prop}

\begin{prop} \label{prop:SNF from minors}
Let $R$ be a PID, and suppose that the $n \times n$ matrix $M$ over $R$ has Smith normal form $\diag(s_1,...,s_n)$.  Then for $1 \leq k \leq n$ we have that $s_1s_2 \cdots s_k$ is equal to the greatest common divisor of all $k \times k$ minors of $M$, with the convention that if all $k \times k$ minors are 0, then their greatest common divisor is 0.   
\end{prop}

We will primarily be interested in determining Smith normal forms over $\Z$, but we will use some results about Smith normal forms over $\Z[t]$ as a computational tool.  When $R=\Z$ we will always assume that the $s_i$ are nonnegative (this can be achieved since $\pm 1$ are the units in $\Z$).  When referring to an abelian group $A = \coker(M)$, we say that $A$ has $k$ \textit{factors} if exactly $k$ of the $s_i$ are different from 1; dually, we write $\ones(A)=k$ if exactly $k$ of the $s_i$ are equal to 1.    

In \cite{MR} Miller and Reiner make the following remarkable conjecture (note that $\Z[t]$ is not a PID):

\begin{conjecture}[\cite{MR}, Conjecture 1.1] \label{conj:Miller-Reiner}
For all differential posets $P$, and for all $n$, the map $U_{n-1}D_n+tI:\Z[t]^{p_n} \to \Z[t]^{p_n}$ has a Smith normal form over $\Z[t]$.
\end{conjecture}

We are interested in Smith normal forms over $\Z[t]$ because of the following very strong consequence:

\begin{prop}[\cite{MR}, Proposition 8.4] \label{prop:snf}
Let $M \in \Z^{n \times n}$ be semisimple and have integer eigenvalues, and suppose $M+tI$ has a Smith normal form over $\Z[t]$.  Then the Smith normal form $S=\diag(s_1,...,s_n)$ of $M$ is given by 
\[
s_{n+1-i}=\prod_{\substack{k \\ m(k) \geq i}} k,
\]
where $m(k)$ denotes the multiplicity of the eigenvalue $k$ of $M$.
\end{prop}

Since $UD$ is semisimple and has integer eigenvalues for any differential poset by Theorem \ref{thm:DP eigenvalues}, if Conjecture \ref{conj:Miller-Reiner} holds for some differential poset $P$, then Proposition \ref{prop:snf} uniquely determines the Smith normal form of $UD$.

Shah showed that Conjecture \ref{conj:Miller-Reiner} is true in the cases of interest to us here:
\begin{theorem}[\cite{Sh}, Corollary 5.2] \label{thm: MR for Y}
For any $r \geq 1$:
\begin{itemize}
\item[a.] Conjecture \ref{conj:Miller-Reiner} holds for $Y^r$.
\item[b.] For all $n$ the down maps $D_n: \Z^{p_n} \to \Z^{p_{n-1}}$ in $Y^r$ are surjective.
\end{itemize}
\end{theorem}

\section{Maps induced between critical groups} \label{sec:maps}
For $\sigma: H \to G$ a group homomorphism and $W$ a representation of $G$, we let $W^{\sigma}$ denote the representation of $H$ given by $h \cdot w := \sigma(h)w$ for all $h \in H, w \in W$.  If $\sigma$ is the inclusion of a subgroup, then $W^{\sigma}=\Res^G_H W$.  If $\sigma$ is an automorphism of $G$, then $W^{\sigma}$ corresponds to the usual notion of \textit{twisting} by $\sigma$.  We extend by linearity to define $W^{\sigma}$ for $W$ a virtual representation.

\begin{theorem} \label{Thm:induced-map}
Let $\sigma: H \inj G$ be an injective group homomorphism and $V$ a faithful representation of $G$, then $\bar{\sigma}: [W]\mapsto [W^{\sigma}]$ is a well-defined group homomorphism $K(V) \to K(V^{\sigma})$.  If $\sigma$ is an isomorphism, then so is $\bar{\sigma}$.
\end{theorem}
\begin{proof}
First observe that the following diagram commutes:
\[
\begin{CD}
R(G) @> \bar{\sigma} >> R(H) \\
@V[\text{$V$}]\cdot(-)VV @VV [\text{$V^{\sigma}$}]\cdot (-) V \\
R(G) @> \bar{\sigma} >> R(H) 
\end{CD}
\]
This is because for any genuine representation $W$, we have
\[
[V^{\sigma}]\cdot [W^{\sigma}]=[V^{\sigma} \otimes W^{\sigma}]=[(V \otimes W)^{\sigma}],
\]
and extending by linearity gives the desired result.  Since $\bar{\sigma}$ preserves virtual dimension, the above diagram restricts to $R_0(G)$ and $R_0(H)$.  The commutativity of the resulting diagram is exactly the condition needed to ensure that the map of quotient groups $K(V) \to K(V^{\sigma})$ is well-defined, and the injectivity of $\sigma$ guarantees that $V^{\sigma}$ is a faithful representation of $H$.  If $\sigma$ is invertible, then it is easy to see that $\bar{\sigma^{-1}}=\bar{\sigma}^{-1}$. 
\end{proof}

\begin{example}
Let $\sigma$ denote the unique outer automorphism of $\mf{S}_6$ (the map $\bar{\sigma}$ is uninteresting for inner automorphisms since $W \cong W^{\sigma}$).  Indexing the irreducible representations of $\mf{S}_6$ by partitions in the usual way, the action of $\bar{\sigma}$ is
\begin{align*}
V_{(5,1)} & \leftrightarrow V_{(2,2,2)}, \\
V_{(2,1^4)} & \leftrightarrow V_{(3,3)}, \\
V_{(4,1,1)} & \leftrightarrow V_{(3,1^3)},
\end{align*}
with the remaining irreducible representations fixed \cite{Wi}.  One can calculate that 
\begin{align*}
K(V_{(5,1)}) & \cong K(V_{(2,2,2)}) \cong (\Z/6 \Z)^2 \oplus \Z/120 \Z,  \\
K(V_{(2,1^4)}) & \cong  K(V_{(3,3)}) \cong \Z/24 \Z \oplus \Z/480 \Z, \\
K(V_{(4,1,1)}) & \cong K(V_{(3,1^3)}) \cong \Z/3 \Z \oplus \Z/90 \Z \oplus \Z/47520 \Z.
\end{align*}
\end{example}

In the case $\sigma: H \hookrightarrow G$, one might have hoped that, in analogy with Proposition \ref{prop:graph covering surj}, the map $\bar{\sigma}: [W] \mapsto [\Res^{G}_H W]$ would be surjective on critical groups; the following example shows that this is not the case for general groups $H \subset G$.
 
\begin{example}
Let $G=D_5$ be the dihedral group of order 10, and let $V$ be the direct sum of a two-dimensional irreducible and the non-trivial one-dimensional irreducible.  This is the complexification of the action of $G$ in $\R^3$ by rotation of a fixed plane and reflection across that plane.  One can calculate (see \cite{G3}, Appendix C) that $K(V)\cong \Z/2 \Z$.  Letting $H=C_5$ be the cyclic subgroup, however, one can show that $K(\Res^G_H V) \cong \Z/5 \Z$.  Thus $\bar{\Res}: K(V) \to K(\Res^G_H V)$ cannot be surjective.

This is a natural counterexample to pick, since $C_5$ has more conjugacy classes than $D_5$, and so $\Res: R(D_5) \to R(C_5)$ cannot be surjective.
\end{example}

There are two classes of groups for which $\bar{\Res}$ can be seen to be surjective for all $V$, both of which will be investigated further throughout the paper.

\begin{prop} \label{prop:surj}
The map $\bar{\Res}:K(V) \to K(\Res^G_H V)$ is surjective if:
\begin{itemize}
\item[(i)] $G$ is abelian,
\item[(ii)] $G=A \wr \mf{S}_n$ and $H=A \wr \mf{S}_m$ for $A$ an abelian group and $m \leq n$.
\end{itemize}
\end{prop}
\begin{proof}
In both cases the map $\Res:R(G) \to R(H)$ is already surjective.  This is clear in case (i); in case (ii) this follows from Theorem \ref{thm: MR for Y} and the fact that $A \wr \mf{S}$ is a differential tower of groups with corresponding differential poset $Y^r$, where $r=|A|$.
\end{proof}

For completeness, we also mention that induction induces a map in the opposite direction:

\begin{prop} \label{prop:ind map}
The map $[W] \mapsto [\Ind^G_H W]$ induces a map on critical groups $\bar{\Ind}: K(\Res V) \to K(V)$.  If $\bar{\Res}$ is a surjection, then $\bar{\Ind}$ is an injection.
\end{prop}
\begin{proof}
To see that $\Ind$ induces a map on critical groups, one easily checks that $(\Ind \: W) \otimes_{\C} V \cong \Ind(W \otimes \Res \: V)$, so that the required diagram commutes.  The second claim follows from the observation that the map $\Ind:R(H) \to R(G)$ is the transpose of the map $\Res:R(G) \to R(H)$ and a standard application of the Snake Lemma.
\end{proof}

\subsection{Cayley graph covering maps}
\label{subsec:cayley}

In this section we investigate the relationship between critical groups of group representations and critical groups of graphs when $G$ is abelian. 

For any finite group $G$, we let $\hat{G}=\Hom(G, \C^{\times})$ denote the Pontryagin dual group.  When $G$ is abelian, all irreducible representations are 1-dimensional, and so $\hat{G}$ is equal to the group of irreducible characters of $G$ under point-wise multiplication.  If $V$ is a faithful representation of an abelian group $G$, then the multiset $\mc{S}_V$ of characters of irreducible components appearing in $V$ generates $\hat{G}$ as a group.  This follows from the standard fact that all irreducible representations of a finite group appear as factors in a sufficiently large tensor power of a fixed faithful representation.

If $G$ is a group with generating multiset $\mc{S}$, the \textit{Cayley graph} $\Cay(G, \mc{S})$ is the directed multigraph with vertex set $G$ and directed edges $g \to gx$ whenever $x \in \mc{S}$.  See Figure \ref{fig:Cayley-graph} for an example of this construction.

\begin{theorem} \label{thm:cay-graph}
For $V$ a faithful representation of an abelian group $G$ the critical groups $K(V)$ and $K(\Cay(\hat{G}, \mc{S}_V))$ can be naturally identified, and the diagram
\begin{center}
$\begin{CD}
K(V) @= K(\Cay(\hat{G}, \mc{S}_V)) \\
@V\bar{\Res}VV @VV\bar{\ph}V \\
K(\Res^G_H V) @= K(\Cay(\hat{H}, \mc{S}_{\Res V}))
\end{CD}$
\end{center}
commutes, where $\bar{\ph}$ is the surjection on critical groups induced by the natural graph covering map $\ph:\Cay(\hat{G},\mc{S}_{V}) \to \Cay(\hat{H}, \mc{S}_{\Res V})$.
\end{theorem} 
\begin{proof}
Define a map $\ph: \hat{G} \sur \hat{H}$ by $\chi \mapsto \Res^G_H \chi$.  Now, $\hat{G}$ and $\hat{H}$ are the vertex sets of the Cayley graphs in question, so $\ph$ induces a graph map $\ph: \Cay(\hat{G}, \mc{S}_V) \sur \Cay(\hat{H}, \mc{S}_{\Res{V}})$ by sending each edge $\chi \to \chi \cdot \psi$ to the edge $\Res \chi \to (\Res \chi)\cdot (\Res \psi)$.  Now, identifying the basis of irreducibles in $R(G)$ and $R(H)$ with the elements of $\hat{G}$ and $\hat{H}$ respectively, it is clear from the definitions that $\tilde{C}_V$ and $\tilde{L}(\Cay(\hat{G},\mc{S}_V))$ define the same linear maps, and that, under this identification $\bar{\ph}$ and $\bar{\Res}$ agree.
\end{proof}

\begin{figure}
\begin{center}
\begin{tikzpicture}[->,>=stealth',shorten >=1pt,auto,node distance=1.6cm,
                thick,main node/.style={circle,draw,font=\Large\bfseries}, baseline=(current bounding box.center)]

  \node[main node] (1) {1};
  \node[main node] (2) [right of=1, fill=gray] {$\zeta$};
  \node[main node] (3) [below right of=2] {$\zeta^2$};
  \node[main node] (4) [below left of=3, fill=gray] {$\zeta^3$};
  \node[main node] (5) [left of=4] {$\zeta^4$};
  \node[main node] (6) [below left of=1, fill=gray] {$\zeta^5$};

  \path
    (1) edge node [green, thick]{} (2)
  		edge node []{} (4)
    (2) edge node []{} (3)
    	edge node []{} (5)
   (3) edge node []{} (4)
  		edge node []{} (6)
    (4) edge node []{} (5)
  		edge node []{} (1)
    (5) edge node []{} (6)
  		edge node []{} (2)
    (6) edge node []{} (1)
  		edge node []{} (3);
\end{tikzpicture}
\hspace{0.2 cm}
$\xrightarrow{\ph}$
\hspace{0.2 cm}
\begin{tikzpicture}[->,>=stealth',shorten >=1pt,auto,node distance=1.6cm,
                thick,main node/.style={circle,draw,font=\Large\bfseries}, baseline=(current bounding box.center)]

  \node[main node] (1) {1};
  \node[main node] (2) [below of=1, fill=gray] {$\zeta^3$};

  \path
    (1) [<->,thick] edge node [right]{2} (2);
   
\end{tikzpicture}
\end{center}
\caption{The Cayley graphs for the case $G=\langle g \rangle \cong \Z/6 \Z$ and $H=\langle g^3 \rangle \cong \Z/ 2 \Z$ with $V$ given by $g \mapsto \diag(\zeta, \zeta^3)$, where $\zeta$ is a primitive sixth root of unity.}
\label{fig:Cayley-graph}
\end{figure}

\begin{remark}
For graph covering maps $\ph:\Gamma \to \Gamma'$, Reiner and Tseng \cite{RT} give an interpretation of the kernel of $\bar{\ph}:K(\Gamma) \sur K(\Gamma')$ as a certain ``voltage graph critical group''.  Thus the identification in Theorem \ref{thm:cay-graph} allows one to describe the kernel of $\bar{\Res}$ in these same terms in the abelian group case. 
\end{remark}
\section{Critical groups and differential posets}\label{sec:gen-perm-rep}

By a \textit{word} of length $2k$, we mean a sequence $w=w_1...w_{2k}$ of $U$'s and $D$'s.  A word $w$ is \textit{balanced} if the number of $U$'s is equal to the number of $D$'s.  When a tower of groups $G_0 \subset G_1 \subset \cdots$ is clear from context, we let $w(\Ind, \Res)$ denote the linear operator $\bigoplus_i R(G_i) \to \bigoplus_i R(G_i)$ defined by replacing the $U$'s in $w$ with $\Ind$ and the $D$'s with $\Res$ and viewing the resulting sequence as a composition of linear operators.  We always assume that induction and restriction are between consecutive groups in the sequence and that $\Res [V]=0$ for $[V] \in R(G_0)$.  Similarly, if $P$ is a differential poset, then we let the linear map $w(U,D): \bigoplus_i \Z^{P_i} \to \bigoplus_i \Z^{P_i}$ be defined as the natural composition of linear operators.  When $w$ is balanced, then for each $i$, $w(\Ind, \Res)$ (resp. $w(U,D)$) restricts to a map $R(G_i) \to R(G_i)$ (resp. $\Z^{P_i} \to \Z^{P_i}$).

\begin{example}
Let $\mf{S}$ be the tower of symmetric groups, and $Y=P(\mf{S})$ denote Young's lattice.  Fix $i \geq 1$, then $w(U,D)=UD$ is a map $\Z^{Y_i} \to \Z^{Y_i}$ and $w(\Ind,\Res)$ is a map $R(\mf{S}_i) \to R(\mf{S}_i)$ sending $[W]\mapsto [\Ind^{\mf{S}_{i}}_{\mf{S}_{i-1}}\Res^{\mf{S}_i}_{\mf{S}_{i-1}} W]$.  Thus $w(\Ind, \Res)[\mathbbm{1}_{\mf{S}_i}]=[\Ind_{\mf{S}_{i-1}}^{\mf{S}_i}\mathbbm{1}_{\mf{S}_{i-1}}]$ is the class of the permutation representation of $\mf{S}_i$.  If we identify $\Z^{Y_i}$ with $R(\mf{S}_i)$ via the differential tower of group structure, then one can check that $w(\Ind, \Res)[\mathbbm{1}_{\mf{S}_i}]\cdot (-)$ and $w(U,D)$ in fact agree as linear maps.  This fact will be explained below.
\end{example}

The following basic facts follow from (\cite{O}, Appendix, Theorem A).
\begin{lemma} \label{lem:tensor induced trivial} Let $G$ be a finite group, and let $H \subset G$ be a subgroup.
\begin{itemize}
\item[a.] The operator $\Ind \Res$ on $R(G)$ has a complete system of orthogonal eigenvectors given by 
\[
\delta^{(g)} := \sum_{V \in \Irr(G)} \chi_V(g) \cdot [V]
\]
as $g$ ranges through a set of conjugacy class representatives of $G$.
\item[b.]  The associated eigenvalue equation is 
\[
\Ind \Res \: \delta^{(g)} = \chi_{\Ind \mathbbm{1}_H}(g) \cdot \delta^{(g)}.
\]
\item[c.] $[\Ind \mathbbm{1}_H] \cdot (-)$ and $\Ind \Res$ are equal as linear maps $R(G) \to R(G)$.
\end{itemize}
\end{lemma}

When $f=\sum_i c_i w^{(i)}$ is a finite nonnegative sum of balanced words, and a differential tower of groups $\mf{G}$ is understood, write $V(f)_n$ for the representation of $G_n$ given by:
\[
f(\Ind, \Res)[\mathbbm{1}_{G_n}]=\sum_i c_i w^{(i)}(\Ind, \Res)[\mathbbm{1}_{G_n}].
\] 
For example, if $f$ consists of the single word $U^kD^k$, and we are working in the tower $\mf{S}$ of symmetric groups, then $V(f)_n=\Ind_{\mf{S}_{n-k}}^{\mf{S}_n} \mathbbm{1}$ is a basic object of study in the representation theory of the symmetric group.  Under the standard characteristic map $\ch: R(\mf{S}_n) \xrightarrow{\sim} \Lambda_n$ between the representation ring of $\mf{S}_n$ and the ring of degree-$n$ symmetric functions, this representation is sent to the complete homogeneous symmetric function $h_{(n-k,1^k)}$ indexed by a ``hook shape".

\begin{prop} \label{prop:IndRes = UD}
Let $\mf{G}: G_0 \subset G_1 \subset \cdots$ be an $r$-differential tower of groups with corresponding differential poset $P=P(\mf{G})$, let $f$ be a finite nonnegative sum of balanced words.  Then, identifying $\Z^{P_n}$ and $R(G_n)$, the maps $f(U,D)_n$ and $[V(f)_n] \cdot (-)$ are equal.  Furthermore the character values of $V(f)_n$ are equal to the eigenvalues of $f(U,D)_n$.
\end{prop}
\begin{proof}
We prove the result for a balanced word $w$, the case for general $f$ then follows by linearity.  By the differential tower of group structure, we have 
\[
DU-UD = rI = \Res \Ind - \Ind \Res
\]
Therefore we can write
\begin{align*}
w(U,D)&=\sum_{i\geq 0} c_iU^iD^i, \\
w(\Ind, \Res)&=\sum_{i \geq 0} c_i \Ind^i \Res^i,
\end{align*}
with the same coefficients $c_i$.  For all $i$, the result for $w'=U^iD^i$ follows from Lemma \ref{lem:tensor induced trivial}, and so the first claim holds by linearity.  

After accounting for the copies of the trivial representation which are present in Proposition \ref{prop:eigenvectors}, that Proposition implies that the eigenvalues of $[V(w)_n]\cdot (-)$ are exactly the character values of $V(w)_n$.  By the first claim we conclude that these must agree with the eigenvalues of $w(U,D)_n$.
\end{proof}

\begin{remark}
Any (not-necessarily-differential) tower of groups gives rise to a graded multigraph in the same way that a differential tower of groups gives rise to the Hasse diagram of a differential poset.  One could write down a statement analogous to Proposition \ref{prop:IndRes = UD} in this context, however we are interested in this special case since the combinatorial and algebraic properties of the up and down maps in differential posets are so strong.
\end{remark}

\subsection{The generalized permutation representation}

The representation \\ $\Ind_{\mf{S}_{n-1}}^{\mf{S}_n} \mathbbm{1}$ of the symmetric group $\mf{S}_n$ is easily seen to be isomorphic to the $n$-dimensional permutation representation, where $\mf{S}_n$ acts by permuting coordinates in $\C^n$.  In \cite{G1}, the second author was able to explicitly compute the critical group for this representation, generalizing Example \ref{ex:perm rep of S4} to arbitrary $n$.  Here we extend that result to a broader class of differential towers of groups:

\begin{theorem} \label{thm:gen-perm-rep}
Let $\mf{G}=G_0 \subset G_1 \subset \cdots$ be an $r$-differential tower of groups such that the associated differential poset $P=P(\mf{G})$ satisfies Conjecture \ref{conj:Miller-Reiner} (such as $\mf{G}=A \wr \mf{S}$ with $A$ abelian of order $r$).  Let $V=V(UD)_n=\Ind_{G_{n-1}}^{G_n} \mathbbm{1}_{G_{n-1}}$.  Then
\[
K(V)=\bigoplus_{i=2}^{p_n} \Z/q_i \Z,
\]
where 
\[
q_i=\prod_{\substack{1 \leq j \leq n \\ \Delta p_j \geq i}} rj.
\]
\end{theorem}
\begin{proof}
By Proposition \ref{prop: size of DTG}, $\dim(V)=|G_n|/|G_{n-1}|=rn$.  By Proposition \ref{prop:IndRes = UD}, the matrix for the map $\tilde{C}$ of multiplication by $rn[\mathbbm{1}]-[V]$ in $R(G_n)$ in the basis $\{V_{\lambda}|\lambda \in P_n\}$ is equal to the matrix for the map $rnI-UD$ in the basis $\{\lambda|\lambda \in P_n \}$.  Since $P$ satisfies Conjecture \ref{conj:Miller-Reiner} by hypothesis, Proposition \ref{prop:snf} and Theorem \ref{thm:DP eigenvalues} together imply that the Smith normal form of $\tilde{C}$ is $\diag(s_1,...,s_{p_n})$ with
\[
s_{p_n+1-i}=\prod_{\substack{0 \leq j \leq n \\ \Delta p_{n-j} \geq i}}(rn-rj). 
\] 
Re-indexing with $j=n-j$ we get that 
\[
s_{p_n+1-i}=\prod_{\substack{0 \leq j \leq n \\ \Delta p_j \geq i}}rj. 
\]
Letting $q_i=s_{p_n+1-i}$, we see that $q_1=0$ since $\Delta p_0 = p_0-p_{-1}=1$; therefore the critical group is given by the direct sum beginning at $i=2$ as in the statement of the theorem.
\end{proof}

Results of Miller \cite{Miller} and Gaetz and Venkataramana \cite{GV} show that $\Delta p_n \geq 2r$ for $n \geq 2$.  In particular, this implies that many of the factors in the products $q_i$ do in fact appear.

\begin{remark}
In \cite{Miller}, Miller has shown that for any differential poset the largest Smith factor of $UD$ agrees with the form predicted by Conjecture \ref{conj:Miller-Reiner}.  Therefore we can conclude that the largest factor of $K(V)$ is $\Z/q_2 \Z$, without assuming that $P(\mf{G})$ satisfies Conjecture \ref{conj:Miller-Reiner}.
\end{remark}

\subsection{The structure of $K(V(f))$}

In this section we investigate the order and subgroup structure of $K(V(f)_n)$ for general finite sums of balanced words $f$.  Although exact formulas for the critical group, like that given in Theorem \ref{thm:gen-perm-rep} for the case $f(U,D)=UD$ remain elusive in general, the results below considerably restrict the structure of $K(V(f)_n)$.

The following proposition of Stanley characterizes eigenspaces for sums of balanced words in a differential poset:

\begin{prop}[\cite{St}, Proposition 4.12] \label{prop:word-eigenspaces}
Let $P$ be an $r$-differential poset and let $f(U,D)$ be a finite sum of balanced words.  Write
\[
f(U,D)=\sum_{j \geq 0} \beta_j(UD)^j
\]
and define
\[
\alpha_i=\sum_{j \geq 0} \beta_j(ri)^j.
\]
Then the characteristic polynomial of $f(U,D)_n:\Z^{P_n} \to \Z^{P_n}$ is given by
\[
\ch f(U,D)_n = \prod_{j=0}^n (x-\alpha_i)^{\Delta p_{n-i}}.
\]
\end{prop}

Proposition \ref{prop:word-eigenspaces} allows us to characterize the order and subgroup structure of critical groups $K(V(f)_n)$.  Since it is clear from the definition that $K(V \oplus \mathbbm{1}) = K(V)$ for all representations $V$, we are free to assume in Proposition \ref{prop:word-eigenspaces} that $\beta_0=0$, and we use this convention in what follows.

\begin{prop} \label{prop:dim of V(f)}
Let $f$ be a nonnegative finite sum of balanced words, and
maintain the notation of Proposition \ref{prop:word-eigenspaces}.  Assume further that $P=P(\mf{G})$ for $\mf{G}$ a differential tower of groups.  Then,
\begin{itemize}
\item[a.] $\dim(V(f)_n)=\alpha_n$, and 
\item[b.] $V(f)_n$ is a faithful representation.
\end{itemize} 
\end{prop}
\begin{proof}
We have $f(U,D)=\sum_{j>0} \beta_j (UD)^j$ and $\alpha_n=\sum_{j > 0} \beta_j (rn)^j$.  Part (a) is immediate from the fact that $V((UD)^j_n)$ is obtained by applying $(\Ind \Res)^j$ to the representation $\mathbbm{1}_{G_n}$, so it has the dimension $[G_n:G_{n-1}]^j=(rn)^j$ by Proposition \ref{prop: size of DTG}.  It is clear from the definition that $\alpha_i < \alpha_n$ for $i \neq n$, thus, since the $\alpha_i$ are the character values of $V(f)_n$, and since $\alpha_n$ has multiplicity $\Delta p_0=1$, we see that $V(f)_n$ is faithful.
\end{proof}

\begin{theorem} \label{thm:structure of K(V(f))}  Let $\mf{G}$ be an $r$-differential tower of groups and let $f(U,D)$ be a nonnegative finite sum of balanced words.  Then, using the notation of Proposition \ref{prop:word-eigenspaces}, we have:
\begin{itemize}
\item[a.] The size of the critical group $K(V(f)_n)$ is given by:
\[
|K(V(f)_n)|=\frac{1}{r^n \cdot n!} \prod_{i=0}^{n-1} (\alpha_n - \alpha_i)^{\Delta p_{n-i}}.
\]
\item[b.] For each $i=1,...,n-1$, the critical group $K(V(f)_n)$ has a subgroup isomorphic to $(\Z/(\alpha_n - \alpha_i) \Z)^{\Delta p_{n-i} -1}$.
\end{itemize}
\end{theorem}
\begin{proof}
This follows from applying Theorem \ref{thm:repeated} with the information about character values given by Proposition \ref{prop:dim of V(f)} and Proposition \ref{prop:IndRes = UD} and the size of $G_n$ as calculated in Proposition \ref{prop: size of DTG}.
\end{proof}

\section{Enumeration of factors in critical groups} \label{sec:factors}

In what follows, when a differential tower of groups and a rank $n$ are understood, we let $\ones(w)$ denote the number of ones in the Smith normal form of $\tilde{C}_{V(w)_n}$, where $w$ is a balanced word.  Then the number of nontrivial factors in the critical group $K(V(w)_n)$ is $p_n-1-\ones(w)$, since $\tilde{C}$ is a $p_n \times p_n$-matrix and there is always a unique zero in the Smith form, by Definition-Proposition \ref{def-prop:critical group}.

The elements of $Y^r$ of rank $n$ are indexed by $r$-tuples $\boldsymbol{\lambda}=(\lambda^{(1)},...,\lambda^{(r)})$ of partitions such that $\sum_i |\lambda^{(i)}|=n$.  Define the \textit{$r$-lexicographic order} on $Y^r$ by first applying lexicographic order on $\lambda^{(1)}$, and then on $\lambda^{(2)}$, and so on; if $|\lambda|>|\mu|$ then we use the convention that $\lambda$ is greater than $\mu$ in the lexicographic order.  We write $\boldsymbol{\lambda} \geq \boldsymbol{\mu}$ to denote this order. 

Our main tool for proving results in this section will be the the characterization of Smith normal form in terms of minors of the matrix, as given in Proposition \ref{prop:SNF from minors}.  If $M$ is a matrix with rows and columns indexed by sets $S,T$ respectively, we let $M_{S',T'}$ denote the submatrix indexed by rows $S' \subset S$ and columns $T' \subset T$.

\begin{theorem} \label{thm:ones of w}
Let $w$ be any balanced word of length $2k \leq 2n$.  Consider the $r$-differential tower of groups $A \wr \mf{S}$, where $A$ is an abelian group of order $r \geq 2$, and the corresponding differential poset $Y^r$.  Then 
\[
\ones(w)=|(Y^r)_{n-k}|=\sum_{\substack{i_1+\cdots+i_r=n-k \\ i_1,...,i_r \geq 0}} \prod_{j=1}^r p(i_j).
\]
In particular, $\ones(w)$ depends only on $r$ and $n-k$, and not on the particular $w$ chosen.
\end{theorem}

We first prove the case $w=U^kD^k$; note that, unlike in Theorem \ref{thm:ones of w}, we do not require $r \geq 2$ in Proposition \ref{prop:ones of UkDk}.

\begin{prop} \label{prop:ones of UkDk}
Let $w=U^kD^k$ with $k \leq n$.  Consider the $r$-differential tower of groups $A \wr \mf{S}$, where $A$ is an abelian group of order $r$, and the corresponding differential poset $P=Y^r$.  Then,
\[
\ones(w)=|(Y^r)_{n-k}|.
\]
\end{prop}
\begin{proof}
Consider the matrix $M=M^{(k)}$ for $U^kD^k : \Z^{P_n} \to \Z^{P_n}$ in the standard basis, ordering the rows and columns from least to greatest in the $r$-lexicographic order, and note that $M$ is symmetric.  We say a partition $\lambda$ has $\ell$ ones if $\lambda_i=1$ for $\ell$ values of $i$. Define subsets of the rows and columns, respectively:
\begin{align*}
S_k &= \{\bs{\lambda} \in P_n | \lambda^{(r)} \text{ has at least $k$ ones}\}, \\
T_k &= \{\bs{\lambda} \in P_n | (\lambda^{(1)})' \text{ has at least $k$ ones}\}. 
\end{align*}
Then we claim that $M_{S_k,T_k}$ is lower unitriangular (see the example in Figure \ref{fig:unitriangular}).  This follows since for each $\bs{\lambda} \in S_k$, the $r$-lexicographically greatest $\bs{\mu}$ which can be obtained from $\bs{\lambda}$ by removing and then adding $k$ boxes from the tuple of Young diagrams can be reached in only one way.  Namely, $k$ of the ones from $\lambda^{(r)}$ are removed, and all of these boxes are added to the largest part of $\lambda^{(1)}$.  The resulting $\bs{\mu}$ is clearly an element of $T_k$.  Now, easy bijections show that $|S_k|=|T_k|=p_{n-k}$.  Thus $M^{(k)}$ and $\tilde{C}_{V(w)_n}$ have a $p_{n-k} \times p_{n-k}$ minor equal to $\pm 1$.  This forces $\ones(w) \geq p_{n-k}$.

\begin{figure} 
\[
\begin{pmatrix}
1 & 2 & 1 & 2 & 2 & 1 & \bs{1} & 0 & 0 & 0 \\
2 & 4 & 2 & 4  & 4 & 2 & 2 & 0 & 0 & 0 \\
1 & 2 & 1 & 2 & 2 & 1 & 1 & 0 & 0 & 0 \\
2 & 4 & 2 & 5 & 5 & 4 & 4 & 1 & 2 & \bs{1} \\
2 & 4 & 2 & 5 & 5 & 4 & 4 & 1 & 2 & 1 \\
1 & 2 & 1 & 4 & 4 & 5 & 5 & 2 & 4 & 2 \\
1 & 2 & 1 & 4 & 4 & 5 & 5 & 2 & 4 & 2 \\
0 & 0 & 0 & 1 & 1 & 2 & 2 & 1 & 2 & 1 \\
0 & 0 & 0 & 2 & 2 & 4 & 4 & 2 & 4 & 2 \\
0 & 0 & 0 & 1 & 1 & 2 & 2 & 1 & 2 & 1 \\
\end{pmatrix} \hspace{0.4in}
\begin{pmatrix}
1 & 1 & 0 & \bs{1} & 0 & 0 & 0 & 0 & 0 & 0 \\
1 & 2 & 1 & 1 & \bs{1} & 0 & 0 & 0 & 0 & 0\\
0 & 1 & 1 & 0 & 1 & 0 & 0 & 0 & 0 & 0 \\
1 & 1 & 0 & 2 & 1 & 1 & \bs{1} & 0 & 0 & 0 \\
0 & 1 & 1 & 1 & 2 & 1 & 1 & 0 & 0 & 0 \\
0 & 0 & 0 & 1 & 1 & 2 & 1 & 1 & \bs{1} & 0 \\
0 & 0 & 0 & 1 & 1 & 1 & 2 & 0 & 1 & \bs{1} \\
0 & 0 & 0 & 0 & 0 & 1 & 0 & 1 & 1 & 0 \\
0 & 0 & 0 & 0 & 0 & 1 & 1 & 1 & 2 & 1 \\
0 & 0 & 0 & 0 & 0 & 0 & 1 & 0 & 1 & 1 \\
\end{pmatrix}
\]
\caption{The matrices $M^{(2)}$ and $M^{(1)}$, in the case $r=2, n=3$ with $w=UDUD=U^2D^2+2UD$.  The diagonal entries of the unitriangular submatrices are shown in bold.}
\label{fig:unitriangular}  
\end{figure}

Now, the map $U^kD^k$ factors through the $(n-k)$-th rank:
\[
\Z^{P_n} \xrightarrow{D^k} \Z^{P_{n-k}} \xrightarrow{U^k} \Z^{P_n}.
\]
Since we know $D$ is surjective for $Y^r$ and $U$ is injective in any differential poset, we see that in fact $\dim \: \ker(U^kD^k) = p_n-p_{n-k}$.  Thus $V(U^kD^k)_n$ has the repeated character value 0 with multiplicity $p_n-p_{n-k}$, and thus a subgroup $(\Z/d \Z)^{p_n-p_{n-k}-1}$ where $d>1$ is the dimension of $V(U^kD^k)_n$ by Theorem \ref{thm:repeated}.  Therefore 
\[
\ones(U^kD^k) \leq (p_n-1)-(p_n-p_{n-k}-1) = p_{n-k}.
\]
\end{proof}

\begin{example} \label{ex:ones of UkDk}
Let $r=1$ in Proposition \ref{prop:ones of UkDk}.  Then $V(U^kD^k)_n= \Ind_{\mf{S}_{n-k}}^{\mf{S}_n} \mathbbm{1}$ is the representation which corresponds to the complete homogeneous symmetric function $h_{(n-k,1^k)}$ under the standard characteristic map (see \cite{M}, Chapter 1).  The proposition implies that the critical group $K(V(U^kD^k)_n)$ has $p(n)-p(n-k)-1$ nontrivial factors.

\textit{Note}: When $r>1$, Proposition \ref{prop:ones of UkDk} still applies, however $V(U^kD^k)_n$ no longer corresponds to the wreath-product analog of $h_{(n-k,1^k)}$ (see \cite{M}, Chapter 1, Appendix B).  That representation is obtained by inducing from $\mf{S}_{n-k} \times A^k$, rather than from $\mf{S}_{n-k}$.
\end{example}

We now return to the proof of the main Theorem.

\begin{proof}[Proof of Theorem \ref{thm:ones of w}.] 
We can write \[
w(U,D)=U^kD^k+c_{k-1}U^{k-1}D^{k-1}+\cdots + c_1UD+c_0I,
\] 
using the relation $DU-UD=rI$; clearly all coefficients $c_i$ are divisible by $r$.  Let $M_w$ be the matrix for $w(U,D)$ using our standard ordered basis, and let $M^{(i)}$ be the matrix for $U^iD^i$, so $M_w=\sum c_i M^{(i)}$.  Using the notation from the proof of Proposition \ref{prop:ones of UkDk}, it is clear that $S_i \subset S_{i-1}$ and $T_i \subset T_{i-1}$ for all $i$, and that the diagonal of $M^{(i)}_{S_i, T_i}$ is above that of $M^{(i-1)}_{S_{i-1},T_{i-1}}$ for all $i$.  Therefore $M_w$ still contains a $p_{n-k} \times p_{n-k}$ lower unitriangular submatrix corresponding to rows and columns $S_k, T_k$.  Since this gives a $p_{n-k} \times p_{n-k}$ minor equal to 1, we know that $\ones(w) \geq p_{n-k}$.

For the other inequality, let $P,Q$ be invertible integer matrices which put $M^{(k)}$ in Smith form: $PM^{(k)}Q=S$.  By Proposition \ref{prop:ones of UkDk}, all but $p_{n-k}$ of the columns of $PM^{(k)}Q$ are divisible by $r$.  Therefore, since the $c_i$ are divisible by $r$, all but $p_{n-k}$ of the columns of $PM_wQ$ are divisible by $r$.  The dimension of $V(w)_n$ is divisible by $r$ as well, thus at most $p_{n-k}$ of the columns of $P\tilde{C}_V Q$ are not divisible by $r$.  Then any minor of a $m \times m$ submatrix with $m > p_{n-k}$ must be divisible by $r$, and so $\ones(w) \leq p_{n-k}$ by Proposition \ref{prop:SNF from minors}.
\end{proof}

\begin{example} \label{ex:r=1 doesn't work}
This example shows that the hypothesis $r \geq 2$ in Theorem \ref{thm:ones of w} is necessary.  Let $w=(UD)^2$, then for $n=7$ one can calculate that $\ones(w)=9 \neq p(7-2)=7$.
\end{example}

We can still give some upper and lower bounds in the $r=1$ case.  For a balanced word $w$ of length $2k$, write
\begin{equation} \label{eq:ci def}
w(U,D)=\sum_{i=0}^k c_i U^iD^i.
\end{equation}
Then define $\ell(w)=\min \{ i | c_i \neq 0 \}$; clearly $0 \leq \ell(w) \leq k$, with equality on the right if and only if $w=U^kD^k$.

\begin{prop}\label{prop:r=1 bounds}
Let $w$ be a balanced word of length $2k \leq 2n$.  Then, working in the tower $\mf{S}$ of symmetric groups, we have
\begin{itemize}
\item[a.] $p(n-k) \leq \ones(w) \leq p(n-\ell(w))$, and
\item[b.] $\ones(w)=p(n-k)$ if $gcd(c_1,...,c_k,\dim(V(w)))>1$, where the $c_i$ are defined as in Equation \ref{eq:ci def}.
\end{itemize}
\end{prop}
\begin{proof}
The argument for the lower bound $p(n-k) \leq \ones(w)$ in the proof of Theorem \ref{thm:ones of w} still holds in the $r=1$ case.  For the upper bound, note that 
\[\ker(w(U,D)) \supset \ker(D^{\ell(w)}:\Z^{Y_n} \to \Z^{Y_{n-\ell(w)}}).
\] 
Thus $\dim \: \ker(w(U,D)) \geq p(n)-p(n-\ell(w))$.  By Theorem \ref{thm:repeated}, this gives the upper bound, proving part (a).

For part (b), notice that if $gcd(c_1,...,c_k,\dim(V(w)))=s>1$, then the argument for the upper bound in the proof of Theorem \ref{thm:ones of w} still applies.
\end{proof}

\begin{example}
Continuing Example \ref{ex:r=1 doesn't work}, we see that 
\[
(UD)^2=UDUD=U(UD+rI)D=U^2D^2+rUD,
\]
and so $\ell((UD)^2)=1$.  Then for $n=7$, Proposition \ref{prop:r=1 bounds} gives that 
\[
7=p(5) \leq \ones(w) \leq p(6) = 11.
\]
In fact we have $\ones(w)=9$, so that we cannot hope for either bound in Proposition \ref{prop:r=1 bounds} to be an equality in general.
\end{example}

\subsection{Smallest factors}

In this section, we give a conjecture on the size and multiplicity of the smallest nontrivial factor in critical groups $K(V(U^kD^k)_n)$.

\begin{conjecture} \label{conj:final-conj}
Working in the differential tower of groups $A \wr \mf{S}$ with $A$ abelian of order $r$, when $k \leq n$ the critical group $K(V(U^kD^k)_n)=K(\Ind_{A \wr \mf{S}_{n-k}}^{A \wr \mf{S}_n} \mathbbm{1})$ is given, as a list of elementary divisors, by:
\[
\left(1^{p_{n-k}}, \left(r^{k} \frac{n!}{(n-k)!}\right)^{p_n-2p_{n-k}+p_{n-2k}},r^{k}\frac{n!}{(n-k)!}e_i \right),
\]
where the exponents involving rank sizes denote multiplicities, and where $e_i$ ranges over the non-unit elementary divisors in the critical group $K(V(D^kU^k)_{n-k})$. 
\end{conjecture}

\begin{remark}
The claim that the multiplicity of 1 as an elementary divisor is $p_{n-k}$ is the content of Proposition \ref{prop:ones of UkDk}.  In the case $r=1$, we were able to prove Conjecture \ref{conj:final-conj}  using explicit row and column operations related to the unitriangular submatrices identified in the proof of Proposition \ref{prop:ones of UkDk} and Theorem \ref{thm:ones of w} which we were unable to generalize to the $r>1$ case.  

In the $r=1$ case, letting $k=n-2$ in Conjecture \ref{conj:final-conj} allows us to explicitly compute that
\begin{align*}
K(\Ind_{\mf{S}_2}^{\mf{S}_n} \mathbbm{1})&=K(U^{n-2}D^{n-2}) \\
&=\left(\Z/\frac{n!}{2}\Z \right)^{p(n)-4} \times \left(\Z/\frac{1}{8}n!(n-2)!(n-2)(n+1) \Z \right),
\end{align*}
where the largest factor is determined by the formula for the the size of the critical group given in Theorem \ref{thm:repeated}.  In the $k=n-3$ case one can again obtain an explicit formula, but this formula depends on $n$ modulo 36, suggesting that a simple general formula is unlikely to exist.
\end{remark}

\section*{Acknowledgements}
The authors wish to thank Vic Reiner for suggesting the analogy between restriction of representations and graph covering maps which led to this project, and for helpful conversations throughout.  The first author wishes to thank the MIT PRIMES program and its organizers Pavel Etingof, Slava Gerovitch, and Tanya Khovanova for their feedback and support throughout the program.

This material is based upon work supported by the National Science Foundation under Grant no. DMS-1519580.

\bibliographystyle{plain}
\bibliography{paper}
\end{document}